\documentclass[12pt]{article}

 \usepackage{graphicx}

\usepackage{amssymb}
 \usepackage{amsthm}
\usepackage{amsmath}

\usepackage[justification=centering,singlelinecheck=false,labelfont=bf,labelsep=period,font=footnotesize]{caption,subfig}

\usepackage{float}

\usepackage{psfrag}

\newtheorem{definition}{Definition}
\newtheorem{example}{Example}
\newtheorem{lemma}{Lemma}
\newtheorem{theorem}{Theorem}

\begin{document}

\title{Derivation, interpretation, and analog modelling  of fractional variable order derivative definition\footnote{Preprint submitted to Applied Mathematical Modelling, January 24, 2013}}
\author{Dominik Sierociuk, Wiktor Malesza and Michal Macias}


\date{\small{Institute of Control and Industrial Electronics, Warsaw University of Technology, Koszykowa 75, Warsaw, Poland
        (e-mail: dsieroci@ee.pw.edu.pl; wmalesza@ee.pw.edu.pl; michal.macias@ee.pw.edu.pl).}}

\maketitle



\begin{abstract}
The paper presents derivation and interpretation of one type of variable order derivative definitions. For mathematical modelling of considering definition the switching and numerical scheme is given. The paper also introduces a numerical  scheme for a variable order derivatives based on matrix approach. Using this approach, the identity of the switching scheme and considered definition is derived. The switching scheme can be used as an interpretation of this type of definition. Paper presents also numerical examples for introduced methods. Finally, the idea and results of analog (electrical) realization of the switching fractional order integrator (of orders $0.5$ and $1$) are presented and compared with numerical approach.
\end{abstract}

%




\section{Introduction}
\label{sec1}

Fractional calculus is a generalization of traditional integer order integration and
differentiation actions onto non-integer order fundamental operator. The idea of such a generalization has been mentioned in 1695 by Leibniz and L'Hospital. In the end of 19th century,  
Liouville and Riemann introduced first definition of fractional derivative. However, only just in late 60' of the 20th
century, this idea drew attention of engineers and mathematicians. Theoretical background of fractional calculus can be found in \cite{old74, pod99, sam87, vinagre2010, Sheng2012}. Basic analysis of fractional differential equations was presented in \cite{Sayevand20124356, Kazem20135498, wiktor2011}.  Fractional calculus was found a very useful tool for modelling behavior of many materials and systems, especially those based on the diffusion processes. One of such devices that can be modeled more efficiently by fractional calculus are ultracapacitors. Models of these electronic storage devices, whose capacity can be even thousands of Farads, based on fractional order models were presented in \cite{Brou:08, Dzie:10, Dzie:11, Quintana:2008}.

Recently, the case when the order is changing in time, started to be intensively developed. The variable fractional order behaviour can be met for example in chemistry (when the properties of the system are changing due to chemical reactions), electrochemistry, and others areas. In~\cite{61}, experimental studies of an electrochemical example of physical fractional variable order  system are presented. In~\cite{58}, the variable order equations were used to describe a history of drag expression.  Papers \cite{61,56,59} present methods for numerical realization of fractional variable order integrators or differentiators.  The fractional variable order calculus also can be used to obtain variable order fractional noise \cite{57}, and to obtain new control algorithms \cite{28}. Some properties of such systems are presented in \cite{27}. In the paper~\cite{21}, the variable order interpretation of the analog realization of fractional orders integrators, realized as domino ladders, was presented. The applications of variable order derivatives and integrals can be found also in signal processing \cite{Sheng2012}.

The rest of the paper is organized as follows. Section~\ref{sec:gl} presents existing generalizations of Grunwald-Letnikov definition of fractional order derivatives. In Section~\ref{sec:prac}, a switching scheme for practical implementation of variable order derivative is given and studied. Section~\ref{sec:prac} presents also a generalization of the matrix approach for switching order and derivation of identity of the switching scheme and the second type of definition. Section~\ref{sec:exampl} presents numerical examples of the proposed methods compared to the analytical solutions. Finally, Section \ref{sec:analog} presents an analog realization of the switched order integrator and comparison of obtained results to the numerical solutions.

\section{Fractional variable order Grunwald-Letnikov type derivatives}\label{sec:gl}

As a base of generalization onto variable order derivative the following definition is taken into consideration:

\begin{definition}\label{def:const}{Fractional constant order derivative is defined as follows:}
\begin{equation*}
_0\mathrm{D}^{\alpha}_t f(t)= \lim_{h\to 0} \frac{1}{h^{\alpha}} \sum^{n}_{r=0} (-1)^r \binom{\alpha}{r} f(t-rh),
\end{equation*}
where $n=\lfloor t/h \rfloor$.
\end{definition}

According to this definition, one obtains: fractional derivatives for $\alpha>0$, fractional integrals for $\alpha<0$, and the original function $f(t)$ for $\alpha=0$. 
For the case of order changing with time (variable order case), three types of definition can be found in the literature \cite{Loren:02}, \cite{Valerio:11}. The first one is obtained by replacing of constant order $\alpha$ by variable order $\alpha(t)$. In that approach, all coefficients for past samples are obtained for present value of the order and is given as follows:
\begin{definition}\label{def:rozn1}{The 1st type of fractional variable order derivative is defined as follows:}
\begin{equation*}
_0\mathrm{D}^{\alpha(t)}_t f(t)= \lim_{h\to 0} \frac{1}{h^{\alpha(t)}} \sum^{n}_{r=0} (-1)^r \binom{\alpha(t)}{r} f(t-rh).
\end{equation*}
\end{definition}
In~Fig.~\ref{fig:def1}, plots of unit step function $1(t)$ derivatives (according to Def.~\ref{def:rozn1}) are presented for $\alpha_1(t) = -1$, $\alpha_2(t) = -2$, and
\begin{equation}\label{eq:a3}
\alpha_3(t) = 
\begin{cases}
-1 & \text{for $0\leq t< 1$,}\\
-2 & \text{for $1\leq t\leq 2$.}
\end{cases}
\end{equation}

\begin{figure}[ht!]
\centering

\psfrag{a3}[l][l]{\scriptsize $-\alpha_3(t)$} 
\psfrag{Da1}[l][l]{\scriptsize $_0\mathrm{D}^{\alpha_1}_t1(t)$}
\psfrag{Da2}[l][l]{\scriptsize $_0\mathrm{D}^{\alpha_2}_t1(t)$}
\psfrag{Da3}[l][l]{\scriptsize $_0\mathrm{D}^{\alpha_3(t)}_t1(t)$}
\psfrag{Time}[l][l]{\scriptsize $t$} 

\includegraphics[scale=0.55]{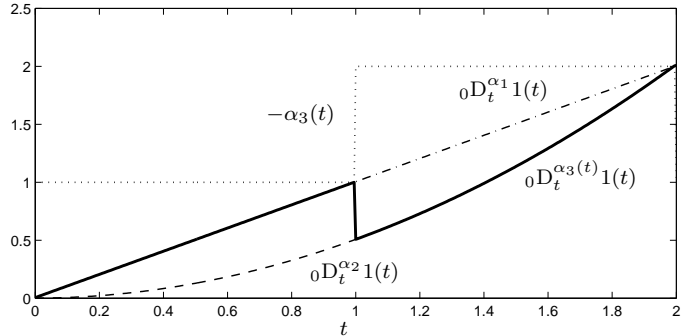}
\caption{Plots of unit step function derivatives with respect to the 1st type derivative (given by Def.~\ref{def:rozn1})}\label{fig:def1}
\end{figure}
The second type of definition assumes that coefficients for past samples are obtained for order that was present for these samples. In this case, the definition has the following form:
\begin{definition}\label{def:rozn2}{The 2nd type of fractional variable order derivative is defined as follows:}
\begin{equation*}
_0\mathrm{D}^{\alpha(t)}_t f(t)= \lim_{h\to 0}  \sum^{n}_{r=0} \frac{(-1)^r}{h^{\alpha(t-rh)}}  \binom{\alpha(t-rh)}{r} f(t-rh).
\end{equation*}
\end{definition}
In~Fig.~\ref{fig:def2}, plots of unit step function derivatives (according to Def.~\ref{def:rozn2}) are presented for $\alpha_1(t) = -1$, $\alpha_2(t) = -2$, and $\alpha_3(t)$ given by~(\ref{eq:a3}).
\begin{figure}[ht!]
\centering

\psfrag{a3}[l][l]{\scriptsize $-\alpha_3(t)$} 
\psfrag{Da1}[l][l]{\scriptsize $_0\mathrm{D}^{\alpha_1}_t1(t)$}
\psfrag{Da2}[l][l]{\scriptsize $_0\mathrm{D}^{\alpha_2}_t1(t)$}
\psfrag{Da3}[l][l]{\scriptsize $_0\mathrm{D}^{\alpha_3(t)}_t1(t)$}
\psfrag{Time}[l][l]{\scriptsize $t$} 

\includegraphics[scale=0.55]{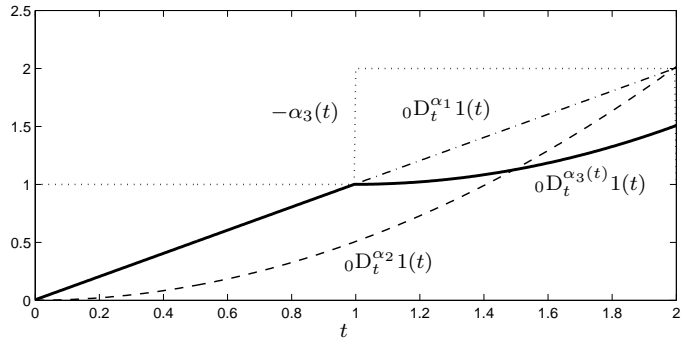}
\caption{Plots of unit step function derivatives with respect to the 2nd type derivative (given by Def.~\ref{def:rozn2})}\label{fig:def2}
\end{figure}
The third definition is less intuitive and assumes that coefficients for the newest samples are obtained respectively for the oldest orders. For such a case, the following definition applies:
\begin{definition}\label{def:rozn3}{The 3rd type of fractional variable order derivative is defined as:}
\begin{equation*}
_0\mathrm{D}^{\alpha(t)}_t f(t)= \lim_{h\to 0}  \sum^{n}_{r=0} \frac{(-1)^r}{h^{\alpha(rh)}} \binom{\alpha(rh)}{r} f(t-rh).
\end{equation*}
\end{definition}
In~Fig.~\ref{fig:def3}, plots of unit step function derivatives (according to Def.~\ref{def:rozn3}) are presented for $\alpha_1(t) = -1$, $\alpha_2(t) = -2$, and $\alpha_3(t)$ given by~(\ref{eq:a3}).
\begin{figure}[ht!]
\centering

\psfrag{a3}[l][l]{\scriptsize $-\alpha_3(t)$} 
\psfrag{Da1}[l][l][1][19]{\scriptsize $_0\mathrm{D}^{\alpha_1}_t1(t)$}
\psfrag{Da2}[l][l]{\scriptsize $_0\mathrm{D}^{\alpha_2}_t1(t)$}
\psfrag{Da3}[l][l]{\scriptsize $_0\mathrm{D}^{\alpha_3(t)}_t1(t)$}
\psfrag{Time}[l][l]{\scriptsize $t$} 

\includegraphics[scale=0.55]{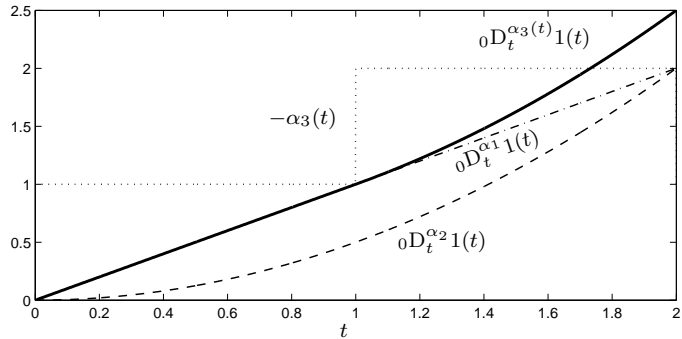}
\caption{Plot of unit step function derivatives with respect to the 3rd type derivative (given by Def.~\ref{def:rozn3})}\label{fig:def3}
\end{figure}

From the comparison of the three types of derivatives, it can be seen from the plots presented above that there are crucial differences between the nature of derivatives of switching order (for constant order derivatives all the definitions yield the same behavior). Namely, in the case of the 1st type derivative, at the switching time instant, the derivative output ``jumps'' from the plot of constant order $\alpha_1$ to the plot of constant order $\alpha_2$  and keeps to follow it. In the case of the 2nd type derivative, at the switching time instant, the derivative output stops fitting the plot of constant order $\alpha_1$ and starts to integrate like at the beginning of the plot of constant order $\alpha_2$ but starting from another initial value. Concerning the 3rd type derivative, at the switching time instant, the plot of switching order derivative stops fitting the plot of constant order $\alpha_1$ and does continue to integrate with constant order $\alpha_2$ (without ``jumping'' to the plot of $\alpha_2$ -- unlike the 1st type derivative) like it is for the plot of constant order $\alpha_2$ at the same time interval, i.e., for $t\geq 1$. 

\section{Practical implementation and numerical scheme of switched (variable) order derivative}\label{sec:prac}
In this section, the routines and schemes for switching order derivative are presented. For simplicity, we start with the simplest case of order switching, namely switching between two real arbitrary constant orders, e.g., $\alpha_1$ and $\alpha_2$. Next, this idea will be generalized for a multiple-switching (variable order) case.

\subsection{Simple-switching order case}
The idea is depicted in Fig.~\ref{fig:chain}, where all the switches $S_i$, $i=1,2$, change their positions depending on an actual value of $\alpha(t)$. If we want to switch from $\alpha_1$ to $\alpha_2$, then, before switching time $T$, we have: $S_1 = b$, $S_2 = a$, and after this time: $S_1 = a$ and $S_2 = b$. At the instant time $T$, the derivative block of complementary order $\bar\alpha_{2}$ is pre-connected on the front of the current derivative block of order $\alpha_1$, where 
\begin{equation}\label{eq:alfa2bar}
\bar\alpha_{2} = \alpha_{2} - \alpha_1.
\end{equation}
If $\bar\alpha_{2}<0$, then $_T\mathrm{D}^{\bar\alpha_{2}}_t$ corresponds to integration of $f(t)$; and, if $\bar\alpha_{2}>0$, then $_T\mathrm{D}^{\bar\alpha_{2}}_t$ corresponds to derivative of $f(t)$, with appropriate order ${\bar\alpha_{2}}$.

\begin{figure}[ht!]
\centering
\includegraphics[scale=1]{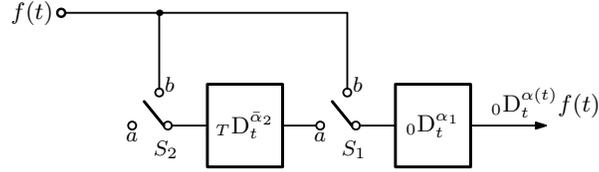}
\caption{Structure of simple-switching order derivative (switching from $\alpha_1$ to~$\alpha_2$)}\label{fig:chain}
\end{figure}

Now, the numerical scheme corresponding to the above derivative switching structure is introduced.
The matrix form of the fractional order derivative is given as follows \cite{matrix_approach,matrix_approach_2}:
\begin{equation*}
\begin{pmatrix}
_0\mathrm{D}^\alpha_0 f({0})\\
_0\mathrm{D}^\alpha_h f({h})\\
_0\mathrm{D}^\alpha_{2h} f({2h})\\
\vdots\\
_0\mathrm{D}^\alpha_{kh} f({kh})\\
\end{pmatrix}= \lim_{h\to 0}W(\alpha,k)\begin{pmatrix}
 f({0})\\
 f({h})\\
 f({2h})\\
\vdots\\
 f({kh})\\
\end{pmatrix},
\end{equation*}
where
\begin{equation}\label{eqn:matrix_appr}
W(\alpha,k)=\begin{pmatrix}
1 & 0 & 0&\dots&0\\
w_{\alpha,1} & 1 & 0&\dots&0\\
w_{\alpha,2} &w_{\alpha,1} & 1 &\dots&0\\
w_{\alpha,3} &w_{\alpha,2} &w_{\alpha,1}&\dots&0\\
\vdots&\vdots&\vdots& \dots&\vdots\\
w_{\alpha,k} &w_{\alpha,k-1} &w_{\alpha,k-2} &\dots&1
\end{pmatrix},
\end{equation}
$W(\alpha,k)\in \mathbb{R}^{(k+1)\times(k+1)}$, $w_{\alpha,i}=\frac{(-1)^i\binom{\alpha}{i}}{h^\alpha}$, and $h=t/k$ is a time step, $k$ is a number of samples.
\begin{lemma}\label{lem1}
For a switching order case, when the switch from order $\alpha_1$ to order ${\alpha_{2}}$ occurs at time $T$, the numerical scheme has the following form:
 \begin{equation*}
\begin{pmatrix}
_0\mathrm{D}^{\alpha(t)}_0 f({0})\\
_0\mathrm{D}^{\alpha(t)}_h f({h})\\
\vdots\\
_0\mathrm{D}^{\alpha(t)}_{(T-1)h} f({Th-h})\\
_0\mathrm{D}^{\alpha(t)}_{Th} f({Th})\\
\vdots\\
_0\mathrm{D}^{\alpha(t)}_{kh} f({kh})
\end{pmatrix}= \lim_{h\to 0} W(\alpha_1,k)W({\bar\alpha_{2}},k,T)\begin{pmatrix}
 f({0})\\
 f({h})\\
\vdots\\
f({Th-h})\\
f({Th})\\
\vdots\\
 f({kh})
\end{pmatrix},
\end{equation*}
where
\begin{equation*}
W({\bar\alpha_{2}},k,T)=\begin{pmatrix}
I_{T,T} & 0_{T,k-T+1}\\
0_{k-T+1,T} &W({\bar\alpha_{2}},k-T)
\end{pmatrix},
\end{equation*}
and 
\begin{equation*}
\alpha(t)=\begin{cases}
\alpha_1 & \text{for $t<T$}, \\
\alpha_1+{\bar\alpha_{2}} & \text{for $t\geq T$}. \end{cases}
\end{equation*}
\end{lemma}
The order ${\bar\alpha_{2}}$, appearing above, is given by relation (\ref{eq:alfa2bar}).
\begin{proof}
The signal incoming to the block of derivative $\alpha_1$ can be described as follows
 \begin{equation*}
\begin{pmatrix}
 f({0})\\
f({h})\\
\vdots\\
 f({Th-h})\\
_{Th}\mathrm{D}^{\bar\alpha_{2}}_{Th} f({Th})\\
\vdots\\
_{Th}\mathrm{D}^{\bar\alpha_{2}}_{kh} f({kh})
\end{pmatrix}= \lim_{h\to 0}W({\bar\alpha_{2}},k,T)\begin{pmatrix}
 f({0})\\
 f({h})\\
\vdots\\
f({Th-h})\\
f({Th})\\
\vdots\\
 f({kh})
\end{pmatrix}.
\end{equation*}
Until time $T$, the input of $\alpha_1$-block  obtains the original function $f(t)$, so in matrix $W({\bar\alpha_{2}},k,T)$ we have an identity matrix. From time step $T$ the input signal is passes through the block of derivative ${\bar\alpha_{2}}$ and we have the sub-matrix $W({\bar\alpha_{2}},k-T)$ that is responsible for starting the ${\bar\alpha_{2}}$ derivative action from time $T$.
That signal passes to the $\alpha_1$-block and has the following matrix form:
\begin{equation*}
\begin{pmatrix}
_0\mathrm{D}^{\alpha(t)}_0 f({0})\\
_0\mathrm{D}^{\alpha(t)}_h f({h})\\
\vdots\\
_0\mathrm{D}^{\alpha(t)}_{(T-1)h} f({Th-h})\\
_0\mathrm{D}^{\alpha(t)}_{Th} f({Th})\\
\vdots\\
_0\mathrm{D}^{\alpha(t)}_{kh} f({kh})
\end{pmatrix}=\lim_{h\to 0}W(\alpha_1,k)\begin{pmatrix}
 f({0})\\
f({h})\\
\vdots\\
 f({Th-h})\\
_{Th}\mathrm{D}^{\alpha_2}_{Th} f({Th})\\
\vdots\\
_{Th}\mathrm{D}^{\alpha_2}_{kh} f({kh})
\end{pmatrix}.
\end{equation*}
\end{proof}
\begin{example}
Consider integration of the unit step function $f(t) = 1(t)$, with the switching variable order $\alpha(t)$ taking value $\alpha_1=-1$until switching time $T = 4$, and after this time, order $\alpha_2=-2$.

Numerical calculations performed for $k=5$ yield:
\[ W(\alpha_1,5) = 
\begin{pmatrix}
1 & 0 & 0 & 0 & 0 & 0\\
1 & 1 & 0 & 0 & 0 & 0\\
1 & 1 & 1 & 0 & 0 & 0\\
1 & 1 & 1 & 1 & 0 & 0\\
1 & 1 & 1 & 1 & 1 & 0\\
1 & 1 & 1 & 1 & 1 & 1
\end{pmatrix}, \quad W(\bar{\alpha}_2,5,4) = 
\begin{pmatrix}
1 & 0 & 0 & 0 & 0 & 0\\
0 & 1 & 0 & 0 & 0 & 0\\
0 & 0 & 1 & 0 & 0 & 0\\
0 & 0 & 0 & 1 & 0 & 0\\
0 & 0 & 0 & 1 & 1 & 0\\
0 & 0 & 0 & 1 & 1 & 1
\end{pmatrix},\]
\[
W(\alpha_1,5)W(\bar{\alpha}_2,5,4) = 
\begin{pmatrix}
1 & 0 & 0 & 0 & 0 & 0\\
1 & 1 & 0 & 0 & 0 & 0\\
1 & 1 & 1 & 0 & 0 & 0\\
1 & 1 & 1 & 1 & 0 & 0\\
1 & 1 & 1 & 2 & 1 & 0\\
1 & 1 & 1 & 3 & 2 & 1
\end{pmatrix},\]

where $\bar{\alpha}_2=-1$.
\end{example}
\subsection{Multiple-switching (variable order) case}
In general case,  when there are many switchings between arbitrary orders, we have the following structure, presented in Fig.~\ref{fig:mchain}.
\begin{figure*}[ht!]
\centering
\includegraphics[scale=1]{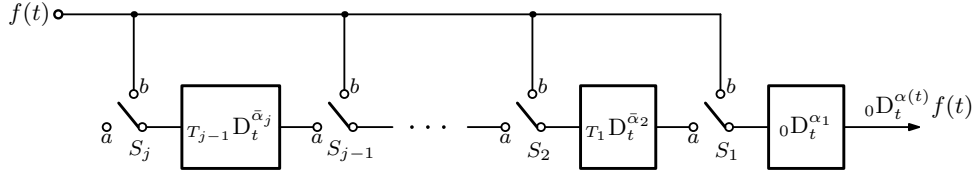}
\caption{Structure of multiple-switching order derivatives}\label{fig:mchain}
\end{figure*}
When we switch from the order $\alpha_{j-1}$ to the order $\alpha_j$ at the switch-time instant $T_{j-1}$, for $j=2,3,\ldots$, we have to set:
\[
S_i = 
\begin{cases}
a & \text{for $i=1,\ldots,j-1$},\\
b & \text{for $i=j$},
\end{cases}
\]
and the pre-connected derivative block (on the front of the previous term) is of the following complementary order:
\[
\bar\alpha_{j} = \alpha_{j} - \alpha_{j-1},
\]
where
\[
\alpha_{j-1} = \alpha_1 + \sum_{k=1}^{j-2}{\bar\alpha_{k+1}}.
\]
The numerical scheme describing the already presented general case of structure allowing to switch between an arbitrary number of orders is given below.
What is very important, the numerical scheme of multiple-switching case (when order is switched in each sample time) is equivalent to the 2nd type of variable order derivative.
\begin{theorem}{Switching order scheme presented in Fig.~\ref{fig:mchain} is equivalent to the 2nd type of variable order derivative (given by Def.~\ref{def:rozn2}).}
\end{theorem}
\begin{proof}
Let us assume that the order of derivative changes with every time step. which gives a variable order derivative, and is given as follows:
\begin{equation}
\alpha_k=\sum_{j=0}^k \bar{\alpha}_j \label{suma},
\end{equation}
where, in this case,  $\alpha_0=\bar{\alpha}_0$ is a value of initial order.
Using Lemma \ref{lem1} the following numerical scheme is obtained
 \begin{equation*}
\begin{pmatrix}
_0\mathrm{D}^{\alpha(t)}_0 f({0})\\
\vdots\\
_0\mathrm{D}^{\alpha(t)}_{kh} f({kh})
\end{pmatrix}= \prod_{j=0}^k W(\bar{\alpha}_j,k,jh)\begin{pmatrix}
 f({0})\\
 \vdots\\
 f({kh})
\end{pmatrix}.
\end{equation*}

The first switching matrices can be described as the following block matrices:
\begin{align*}
W(\bar{\alpha}_0,k,0)W(\bar{\alpha}_1,k,1)&=\left( \begin{array}{c|c}
1 & 0_{1,k}\\ \hline
R(\bar{\alpha}_0,1) & W(\bar{\alpha}_0,k-1)
\end{array}\right)
\left( \begin{array}{c|c}
1 & 0_{1,k}\\ \hline
0 & W(\bar{\alpha}_1,k-1)
\end{array}\right) \\
&= \left( \begin{array}{c|c}
1 & 0_{1,k}\\ \hline
R(\bar{\alpha}_0,1) & W(\bar{\alpha}_0+\bar{\alpha}_1,k-1)
\end{array}\right),
\end{align*}
where 
\begin{equation*}
R(\bar{\alpha},i)=\begin{pmatrix}
w_{\bar{\alpha},i}\\
w_{\bar{\alpha},i+1}\\
\vdots 
\end{pmatrix}
\end{equation*}
is a vector with coefficients given by~(\ref{eqn:matrix_appr}).
For a switching in the next sample time, we obtain the following numerical scheme:
{\setlength\arraycolsep{2.5pt}
\begin{align*}
&W(\bar{\alpha}_0,k,0)W(\bar{\alpha}_1,k,1)W(\bar{\alpha}_2,k,2)\\
&= \left( \begin{array}{cc|c}
1 & 0 & 0_{1,k-1}\\ 
w_{\bar{\alpha}_0,1} &1 & 0_{1,k-1}\\\hline
R(\bar{\alpha}_0,2)& R(\bar{\alpha}_0+\bar{\alpha}_1,1) & W(\bar{\alpha}_0+\bar{\alpha}_1,k-2)
\end{array}\right)\left( \begin{array}{cc|c}
1 & 0 & 0_{1,k-1}\\
0 & 1 & 0_{1,k-1}\\ \hline
0 &0 & W(\bar{\alpha}_2,k-1)
\end{array}\right)\\
&= \left( \begin{array}{cc|c}
1 & 0 & 0_{1,k-1}\\ 
w_{\bar{\alpha}_0,1} &1 & 0_{1,k-1}\\\hline
R(\bar{\alpha}_0,2)& R(\bar{\alpha}_0+\bar{\alpha}_1,1) & W(\bar{\alpha}_0+\bar{\alpha}_1+\bar{\alpha}_2,k-2)
\end{array}\right).
\end{align*}}
In the case of $k-th$ switchings of order, we have the following form of the switching matrix $W(\bar{\alpha}_0,k,0)W(\bar{\alpha}_1,k,1)\cdots W(\bar{\alpha}_k,k,k)$, i.e.,
\begin{equation*}
\begin{pmatrix}
1 & 0 & 0&\dots&0\\
{w_{\bar{\alpha}_0,1}}& 1 & 0&\dots&0\\
{w_{\bar{\alpha}_0,2}} &{w_{\bar{\alpha}_0+\bar{\alpha}_1,1}} & 1 &\dots&0\\
{w_{\bar{\alpha}_0,3}}&{w_{\bar{\alpha}_0+\bar{\alpha}_1,2}}&{w_{\bar{\alpha}_0+\bar{\alpha}_1+\bar{\alpha}_2,1}} &\dots&0\\
\vdots&\vdots&\vdots& \dots&\vdots\\
{w_{\bar{\alpha}_0,k}} &{w_{\bar{\alpha}_0+\bar{\alpha}_1,k-1}}&{w_{\bar{\alpha}_0+\bar{\alpha}_1+\bar{\alpha}_2,k-2}} &\dots&1
\end{pmatrix}
\end{equation*}
or shortly, using the relationship~(\ref{suma})
\begin{equation*}
\begin{pmatrix}
1 & 0 & 0&\dots&0&0\\
w_{{\alpha}_0,1} & 1 & 0&\dots& 0& 0\\
w_{{\alpha}_0,2} &w_{{\alpha}_1\,1} & 1 &\dots&0&0\\
w_{{\alpha}_0,3} &w_{{\alpha}_1,2} &w_{{\alpha}_2,1}&\dots&0&0\\
\vdots&\vdots&\vdots& \dots&\vdots&\vdots\\
w_{{\alpha}_0,k-1} &w_{{\alpha}_1,k-2} &w_{{\alpha}_2,k-3}&\dots&1&0\\
w_{{\alpha}_0,k} &w_{{\alpha}_1,k-1} &w_{{\alpha}_2,k-2} &\dots&w_{{\alpha}_{k-1},1}&1
\end{pmatrix}.
\end{equation*}
The coefficients of the matrix above are identical with the coefficients given by~Def.~\ref{def:rozn2}, which completes the proof.
\end{proof}
\section{Numerical examples}\label{sec:exampl}
This section contains numerical examples of the proposed methods, computed in Matlab/Simulink environment, using the dedicated numerical routines~\cite{vfodif}, developed by the authors. 
\begin{example}\label{ex:integral_1_2_3_1}
Let us consider a variable order integration of unit step function $f(t)=1(t)$ with the following sequence of integer orders $\Lambda = \{-1,-2,-3,-1\}$ switched every one second, i.e.,
\begin{align*}
\alpha(t) &= 
\begin{cases}
-1 & \text{for\; $0\leq t< 1$},\\
-2 & \text{for\; $1\leq t< 2$},\\
-3 & \text{for\; $2\leq t< 3$},\\
-1 & \text{for\; $3\leq t\leq 4$}.
\end{cases}
\end{align*}
The analytical solution of such integration is the following
\begin{align*}
_0\mathrm{D}^{\alpha(t)}_t f(t) &= \int_0^t{\frac{f(\tau)(t-\tau)^{\alpha(\tau)-1}}{\Gamma(\alpha(\tau))}\mathrm{d}\tau}\\
&=
\begin{cases}
t & \text{for\; $0\leq t< 1$},\\
\frac{1}{2}t^2 - t + \frac{3}{2} & \text{for\; $1\leq t< 2$},\\
\frac{1}{6}t^3 - t^2 + 3t - \frac{11}{6} & \text{for\; $2\leq t< 3$},\\
\frac{1}{2}t^2 - \frac{1}{2}t - \frac{1}{3} & \text{for\; $3\leq t\leq 4$}.
\end{cases}
\end{align*}
The plot of integration result is presented in~Fig.~\ref{fig:integral_1_2_3_1}.
\begin{figure}[ht!]
\centering

\psfrag{a}[l][l]{\scriptsize $-\alpha(t)$} 
\psfrag{f}[l][l]{\scriptsize $_0\mathrm{D}^{\alpha(t)}_t1(t)$}
\psfrag{Time}[l][l]{\scriptsize $t$} 

\includegraphics[scale=0.55]{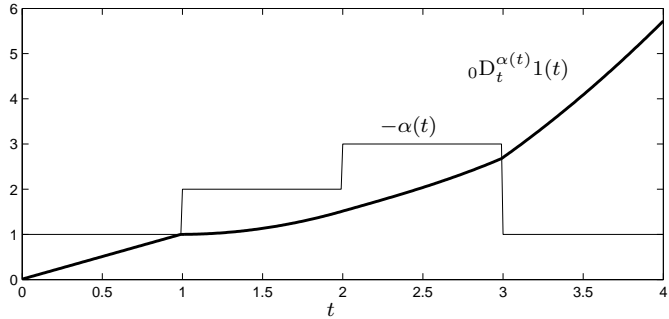}
\caption{Plot of the variable order integral of unit step function~(Ex.~\ref{ex:integral_1_2_3_1})}\label{fig:integral_1_2_3_1}
\end{figure}
The difference between analytical and numerical solution, for different integration steps $h$, is depicted in~Fig.~\ref{fig:integral_1_2_3_1_diff}.

\begin{figure}[ht!]
\centering

\psfrag{k1}[l][l]{\scriptsize $h=0.05$} 
\psfrag{k2}[l][l]{\scriptsize $h=0.01$}
\psfrag{k3}[l][l]{\scriptsize $h=0.005$} 
\psfrag{t}[l][l]{\scriptsize $t$} 

\includegraphics[scale=0.55]{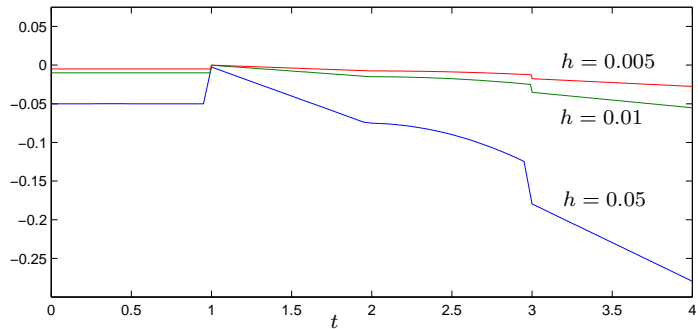}
\caption{Errors between analytical and numerical solution for different integration steps~(Ex.~\ref{ex:integral_1_2_3_1})}\label{fig:integral_1_2_3_1_diff}
\end{figure} 

%
%

\end{example}

\begin{example}\label{ex:integral_fractional}
Let us consider a variable order integration of constant function $f(t)=1(t)$ with the following sequence of fractional orders $\Lambda = \{-0.4,-1.8,-0.5,-2.5\}$ switched with every one second, i.e.,
\begin{align*}
\alpha(t) &= 
\begin{cases}
-0.4 & \text{for\; $0\leq t< 1$},\\
-1.8 & \text{for\; $1\leq t< 2$},\\
-0.5 & \text{for\; $2\leq t< 3$},\\
-2.5 & \text{for\; $3\leq t\leq 4$}.
\end{cases}
\end{align*}
The analytical solution of such integration is the following
\begin{align*}
_0\mathrm{D}^{\alpha(t)}_t f(t)&\approx 
\begin{cases}
D_1 & \text{for\; $0\leq t< 1$},\\
D_2 & \text{for\; $1\leq t< 2$},\\
D_3 & \text{for\; $2\leq t< 3$},\\
D_4 & \text{for\; $3\leq t\leq 4$},
\end{cases}
\end{align*}
where
\begin{align*}
D_1 &= 1.127 t^{0.4},\\ 
D_2 &= 0.597(t-1)^{1.8} + 1.127\left(t^{0.4} - (t-1)^{0.4}\right),\\
D_3 &= 1.128\sqrt{t-2} + 0.597\left((t-1)^{1.8} - (t-2)^{1.8}\right)\\ 
&+ 1.127\left(t^{0.4} - (t-1)^{0.4}\right),\\
D_4 &= 0.3(t-3)^{2.5} + 0.597\left((t-1)^{1.8} - (t-2)^{1.8}\right)\\ 
&+ 1.128\left(\sqrt{t-2} - \sqrt{t-3}\right) + 1.127\left(t^{0.4} - (t-1)^{0.4}\right).
\end{align*}
The plot of integration result is presented in~Fig.~\ref{fig:integral_fractional}.
\begin{figure}[ht!]
\centering

\psfrag{a}[l][l]{\scriptsize $-\alpha(t)$} 
\psfrag{f}[l][l]{\scriptsize $_0\mathrm{D}^{\alpha(t)}_t1(t)$}
\psfrag{Time}[l][l]{\scriptsize $t$} 

\includegraphics[scale=0.55]{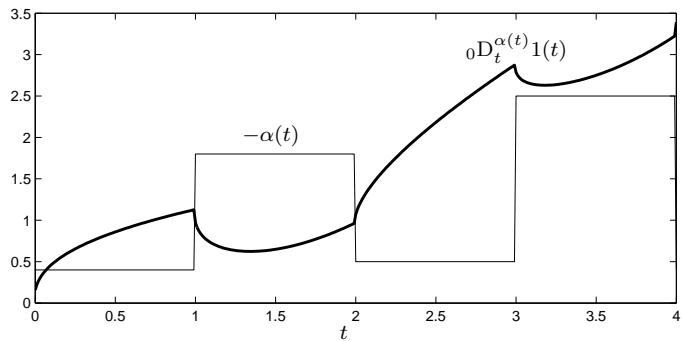}
\caption{Plot of the variable order integral of constant function~(Ex.~\ref{ex:integral_fractional})}\label{fig:integral_fractional}
\end{figure}

%
\end{example}

\section{Analog realization of the second type of fractional variable order derivative}
\label{sec:analog}

\subsection{Analog realization of the half order integral}\label{sec:domino_approx}

In this paper, the following method of half order integrator implementation, introduced in~\cite{SierociukMMAR11} and meticulously investigated in~\cite{21,petras2012}, will be used.
The scheme of this method is presented in Fig.~\ref{fig:half}. 
%
%
%
%


\begin{figure}[!ht]
\centering
\includegraphics[scale=0.8]{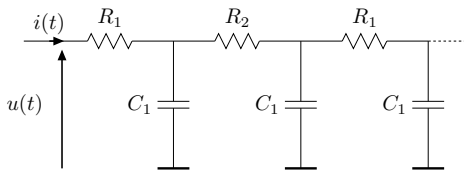}
\caption{Domino ladder type analog model of half order integrator}
\label{fig:half}
\end{figure}

Based on the observation, the values of $R_1$ and $C_1$ are chosen in order to satisfy the required low frequency limit. Value $R_2$ is chosen in order to satisfy the condition $R_1/R2\approx 0.25$, and the realization length $m$ (number of branches) is chosen in order to satisfy required frequency range of the approximation.

\subsection{Experimental setup}

An analog realization of switching system directly based on the switching scheme given in Fig.~\ref{fig:chain} is presented in Fig.~\ref{fig:circdef}. This scheme can be easily extended to the switching rule presented in Fig.~\ref{fig:mchain}.
%
The experimental setup was prepared in order to modelling simple-switching case. The structure of the half order integrator, used in the setup, is based on the domino ladder approximation presented in Section \ref{sec:domino_approx}. However, the first order integrator is realized according to traditional scheme based on capacitor. Because, in both cases of realizations with operational amplifiers, the signals with inverted polarization were obtained. This required amplifiers with gain $-1$ for each integrator (amplifiers $A_2$ and $A_4$). Integrators based on amplifiers $A_1$ and $A_3$ contain gain resistors $R=9.65k\Omega$ and impedances $Z_1$ and $Z_2$ chosen according to the type of realized orders. As a realization of switches $S_1$ and $S_2$ integrated analog switches DG303 were used. In order to obtain impedance order equal to $-0.5$, the domino ladder approximation with the following elements $R_{1}=2k\Omega$, $R_{2}=8.2k\Omega$, $C_{1}=470nF$ and of implementation length $m=100$, was used. The experimental circuit was connected to the  dSPACE DS1104 PPC card with a PC. 
\begin{figure}[!ht]
\centering
\includegraphics[scale=0.7]{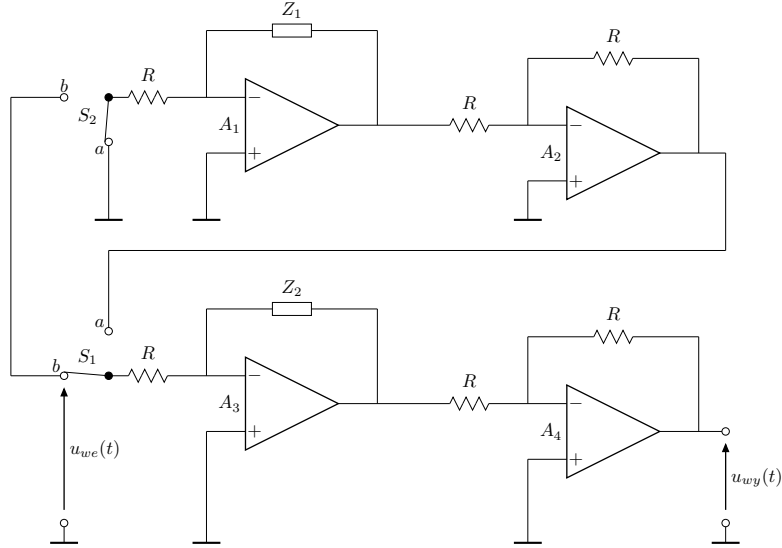}
\caption{Analog realization of the 2nd type of fractional variable order integral}
\label{fig:circdef}
\end{figure}
\subsection{Experimental results}


\subsubsection{Switching between order $\alpha=-0.5$ and $\alpha=-1$}
For a case of switching between order $\alpha=-0.5$ and $\alpha=-1$, in the structure presented in Fig.~\ref{fig:circdef}, both impedance $Z_1$ and $Z_2$ contains domino ladders structures each of order $-0.5$. This implies that operational amplifiers $A_1$ and $A_3$ realize half order integrators circuits. Before switching, only one half order integrator is used and the whole system possesses $-0.5$ order. After switching, both half order integrators are connected in series, which results that the system order is equal to $-1$.

The experimental results of integrator with switched orders from $\alpha=-0.5$ to $\alpha=-1$ compared with numerical results are presented in Fig.~\ref{analog_out_05_1}. The difference between numerical realization of the 2nd type derivative definition and its analog implementation is presented in Fig.~\ref{analog_err_05_1}. The sample time for all measurements was chosen as $0.01$ sec and input signal was $u(t)=0.01 \cdot 1(t)$.
The parameters of analog models were obtained by identification based on time domain responses, separately for both orders.
Obtained transfer function for the case of order $\alpha=-0.5$ (single half order integrator) is
\begin{equation*}
G_1(s)=\frac{1}{0.058s^{0.5}}.
\end{equation*}
and for the case of order $\alpha=-1$ (both half order integrators in series connection)
\begin{equation*}
G_2(s)=\frac{1}{0.041s},
\end{equation*}
The switching time was equal to $1$ sec.
\begin{figure}[!ht]
\centering
\includegraphics[width=0.7\textwidth]{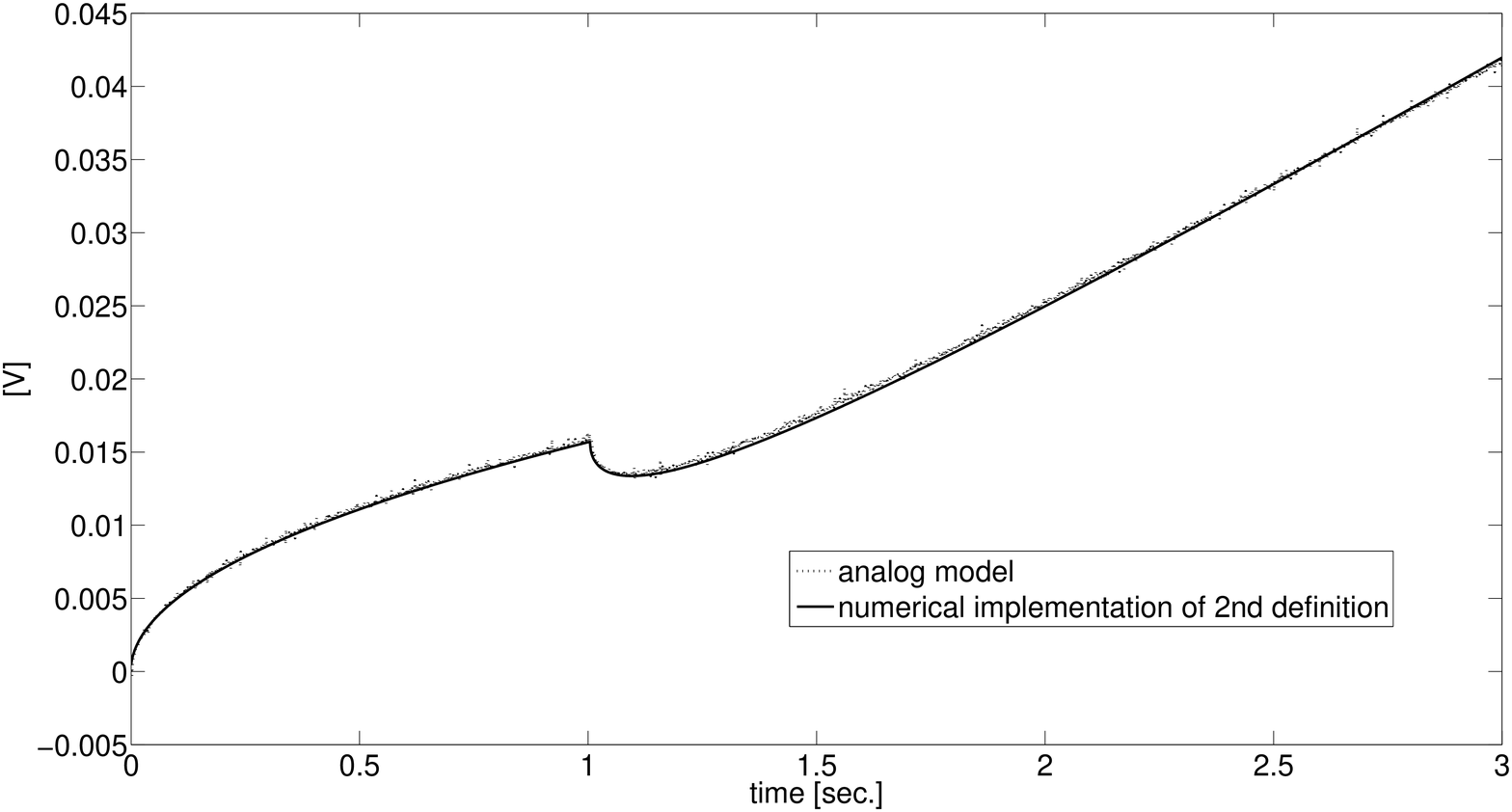}
\caption{Results of analog and numerical implementation of switching order derivative from $\alpha=-0.5$ to $\alpha=-1$ }
\label{analog_out_05_1}
\end{figure}
\begin{figure}[!ht]
\centering
\includegraphics[width=0.7\textwidth]{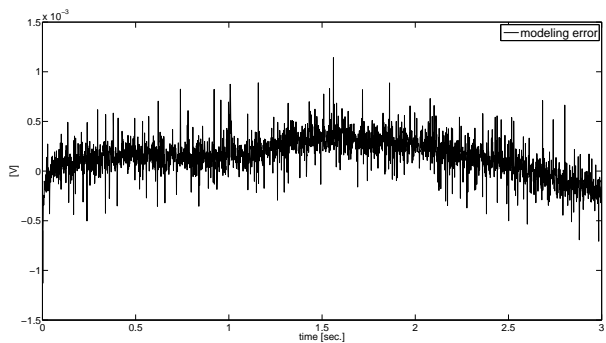}
\caption{Difference between analog and numerical implementation of switching order derivative from $\alpha=-0.5$ to $\alpha=-1$ }
\label{analog_err_05_1}
\end{figure}
\subsubsection{Switching between order $\alpha=-0.5$ and $\alpha=-1.5$}
In this case, impedance $Z_1$ is the first order capacitor, and impedance $Z_2$ is a $-0.5$ order domino ladder. 
The experimental results of integrator with order switching from $\alpha=-0.5$ to $\alpha=-1.5$, compared to the numerical results, are presented in Fig.~\ref{analog_out_05_15}. The difference between numerical realization of the 2nd type derivative definition and its analog implementation is presented in Fig.~\ref{analog_err_05_15}. 
By identification the following models were obtained.
For order $\alpha=-0.5$
\begin{equation*}
G_1(s)=\frac{1}{0.68s^{0.5}},
\end{equation*}
and for order $\alpha=-1.5$
\begin{equation*}
G_2(s)=\frac{1}{0.0295s^{1.5}}.
\end{equation*}
The switching time was equal to $1$ sec.
\begin{figure}[!ht]
\centering
\includegraphics[width=0.7\textwidth]{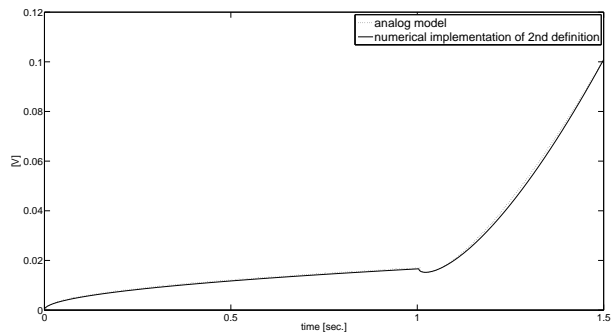}
\caption{Results of analog and numerical implementation of switching order derivative from $\alpha=-0.5$ to $\alpha=-1.5$ }
\label{analog_out_05_15}
\end{figure}
\begin{figure}[!ht]
\centering
\includegraphics[width=0.7\textwidth]{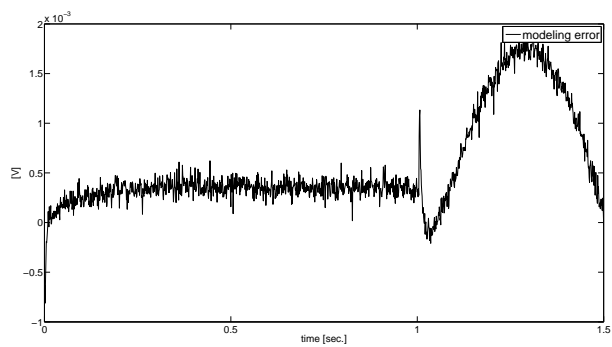}
\caption{Difference between analog and numerical implementation of switching order derivative from $\alpha=-0.5$ to $\alpha=-1.5$ }
\label{analog_err_05_15}
\end{figure}
\subsubsection{Switching between order $\alpha=-1$ and $\alpha=-1.5$}
In this case, impedance $Z_1$ is a domino ladder of order $-0.5$, and impedance $Z_2$ is the first order capacitor.
The experimental results of integrator with order switching from $\alpha=-1$ to $\alpha=-1.5$, compared to the numerical results, are presented in Fig.~\ref{analog_out_1_15}. The difference between numerical realization of the 2nd type derivative definition and its analog implementation is presented in Fig.~\ref{analog_err_1_15}. 
By identification the following models were obtained.
For order $\alpha=-1$
\begin{equation*}
G_1(s)=\frac{1}{0.038s},
\end{equation*}
and for order $\alpha=-1.5$
\begin{equation*}
G_2(s)=\frac{1}{0.027s^{1.5}}.
\end{equation*}
The switching time was equal to $0.7$ sec.
\begin{figure}[!ht]
\centering
\includegraphics[width=0.7\textwidth]{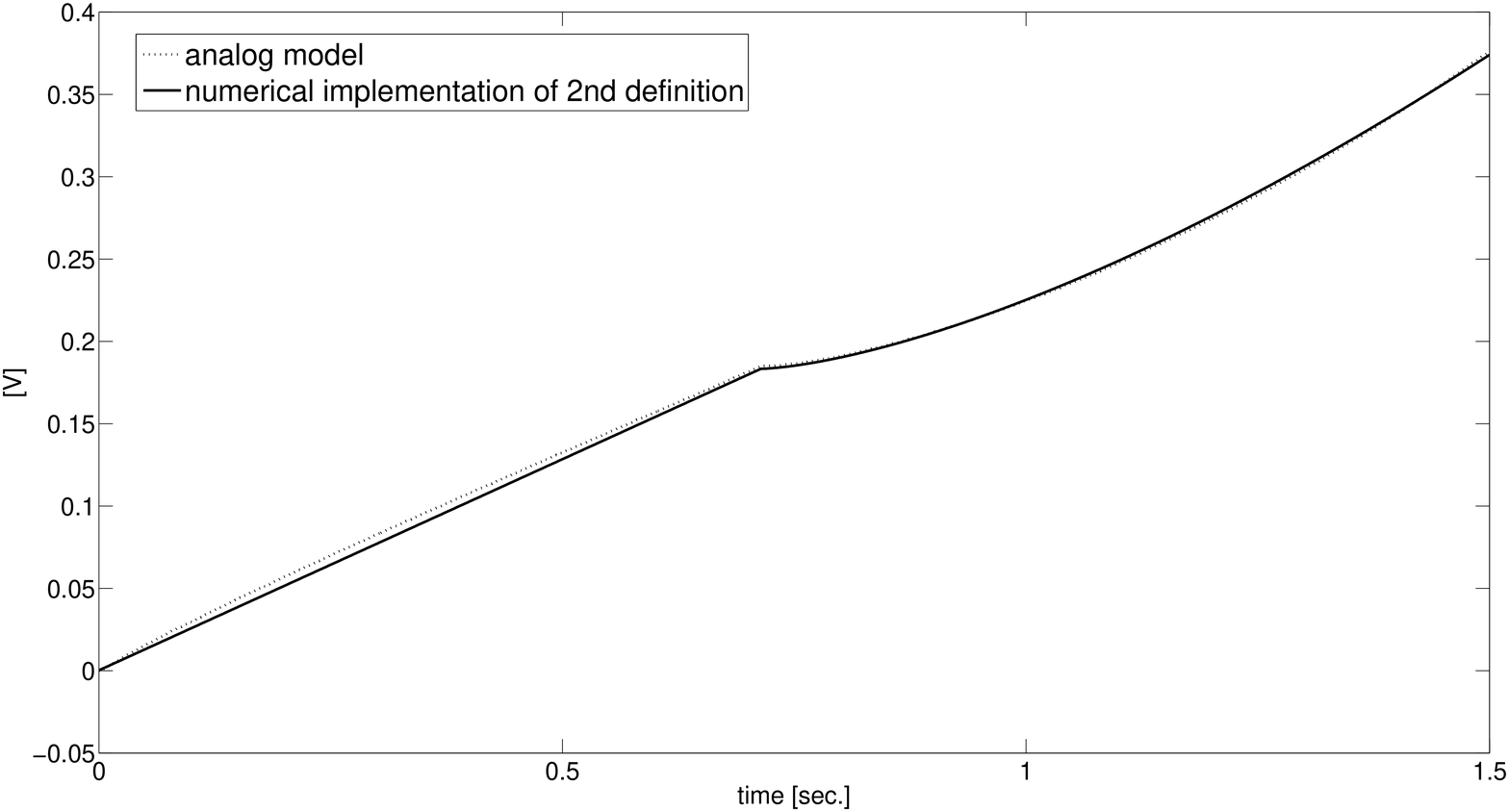}
\caption{Results of analog and numerical implementation of switching order derivative from $\alpha=-1$ to $\alpha=-1.5$ }
\label{analog_out_1_15}
\end{figure}
\begin{figure}[!ht]
\centering
\includegraphics[width=0.7\textwidth]{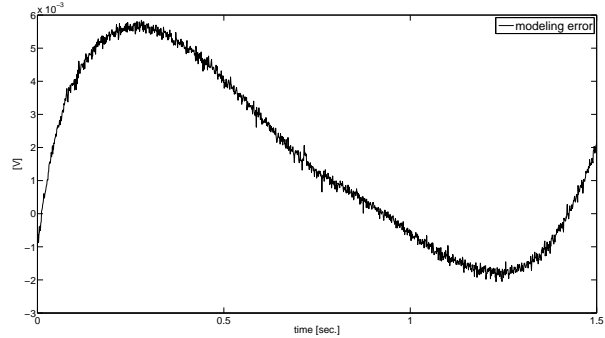}
\caption{Difference between analog and numerical implementation of switching order derivative from $\alpha=-1$ to $\alpha=-1.5$ }
\label{analog_err_1_15}
\end{figure}
\section{Conclusions}
\label{sec:conclusions}

The paper presented a switching scheme for the second type of variable order derivative. The numerical scheme, based on matrix approach, for that switching scheme was introduced and investigated. It was shown that obtained numerical scheme is equivalent to the 2nd type of fractional variable order derivative. This switching scheme can also be used as an interpretation of the second type of the definition. It is also worth to notice that the rule given by the switching scheme can itself be a definition of a variable order derivative and the relation given by Def.~\ref{def:rozn2} is just a consequence of this scheme. Presented interpretation allows to better understand behaviour of the definition and, in general, switching process in variable order systems. Based on this, it can give rise to more appropriate choice of definitions type, depending on particular application. Additionally, an analog modelling method of switched order integrator was introduced and examined. Obtained results were compared with the numerical implementation and show high accuracy of the introduced method. 

\section*{Acknowledgment}

This work was supported by the Polish National Science Center with the decision number DEC-2011/03/D/ST7/00260



\end{document}